\documentclass[12pt,reqno]{amsart}
\usepackage{hyperref}
\usepackage{amsmath,amssymb}
\usepackage[english]{babel}
\usepackage{caption}
\usepackage{graphicx}
\usepackage{float}
\usepackage{pgf,tikz}
\usepackage{upgreek}
\usepackage{mathrsfs}

\newcommand \Hilb{\operatorname{Hilb}}

\newtheorem{theorem}{Theorem}[section]
\newtheorem{definition}[theorem]{Definition}
\newtheorem{lemma}[theorem]{Lemma}
\newtheorem{proposition}[theorem]{Proposition}

\newtheorem{corollary}[theorem]{Corollary}

\begin{document}
\title{HILBERT SERIES OF BINOMIAL EDGE IDEALS}
\author{Arvind Kumar}
\email{arvkumar11@gmail.com}

\author{Rajib Sarkar}
\email{rajib.sarkar63@gmail.com}
\address{Department of Mathematics, Indian Institute of Technology
Madras, Chennai, INDIA - 60036}

\begin{abstract}
Let $G$ be a finite simple graph on $n$ vertices and $J_G$ denote the
corresponding binomial edge ideal in the polynomial ring $S = K[x_1, \ldots, x_n, y_1,
\ldots, y_n].$ In this article, we compute the Hilbert series of binomial
edge ideal of decomposable graphs in terms of Hilbert series of its indecomposable
subgraphs. Also, we compute the Hilbert series of binomial edge ideal of join
of two graphs and as a consequence we obtain the Hilbert series of complete
$k$-partite graph, fan graph, multi-fan graph  and wheel graph.
\end{abstract}
\maketitle

\section{Introduction}
Let $G$ be a finite simple graph on the vertex set $[n]$. Herzog et al. in \cite{HH1}
and Ohtani, independently  in \cite{oh}, introduced the notion of
binomial edge ideal corresponding to a finite simple graph.  Let
$S=K[x_1, \ldots, x_n,y_1, \ldots, y_n]$, where $K$ is a field. 
The binomial edge ideal of the graph $G$ is  $J_G =(x_i y_j - x_j
y_i : \{i,j\} \in E(G), \; i <j)$. Researchers have been trying to relate 
the algebraic properties of $J_G$ with the combinatorial properties of $G$, 
see for example \cite{AB,her1,HH1,JA1,JNR,KM6,FM}. 
While the Castelnuovo-Mumford regularity of binomial edge ideals of
several classes of graphs are known, not much is known about other
invariants such as Betti numbers, Hilbert series and multiplicity.
Betti numbers and Hilbert series of the binomial edge ideals of cycles
were computed by Zafar and Zahid in \cite{SZ}.  Mohammadi and Sharifan
studied the Hilbert series  of quasi cycles in \cite{FM}. In
\cite{MR3195706}, Schenzel and Zafar computed dimension, depth,
multiplicity and Betti numbers of complete bipartite graphs. 
The aim of this article is to compute the Hilbert series of 
binomial edge ideals of graphs in terms of the Hilbert series of
certain subgraphs.

A graph $G$ is said to be decomposable if there exist induced subgraphs $G_1$ and $G_2$ such that  $G= G_1 \cup G_2$, $V(G_1)\cap V(G_2)=\{v\}$ and $v$ is 
a free vertex of $G_1$ and $G_2$. A graph $G$ is indecomposable, if it is not decomposable.
Upto permutation, $G$
has  a unique decomposition into indecomposable subgraphs, i.e. there exist indecomposable  subgraphs $G_1,\ldots,G_r$  of $G$ with 
$G=G_1\cup \cdots \cup G_r$ such that for each $i \neq j$, either 
$V(G_i) \cap V(G_j) = \phi$ or $V(G_i) \cap V(G_j) =\{v_{i,j}\}$ and $v_{i,j}$ is a free vertex 
of $G_i$ and $G_j$. In Section $3$, we obtain the Hilbert 
series and multiplicity  of a decomposable graph in terms of the Hilbert series of
its indecomposable subgraphs (Theorem \ref{3.3}). As consequences we obtain the Hilbert series
and multiplicity of Cohen-Macaulay closed graphs and $k$-handle lollipop graphs. 

Let $H$ and $H'$ be two graphs  with the vertex sets $[p]$ and $[q]$,
respectively. The join of $H$ and $H'$, denoted by $H*H'$  is 
the graph with vertex set $[p] \sqcup [q]$ and the edge set
$E(H*H')= E(H) \cup E(H') \cup \{\{i,j\}| i \in [p], j \in [q]\}$. 
Our aim is to compute the Hilbert series of  $H*H'$ in terms of
the Hilbert series of $H$ and $H'$. First, we treat the case
when both the graphs are disconnected (Theorem \ref{4.3}).
In order to compute the Hilbert series of join of two arbitrary graphs, it is necessary to understand
the structure of some other intermediate graphs. We define a
product operation of a graph with the complete graph and study the minimal primes and Hilbert
series of the resulting graph. We further compute the Hilbert series of several intermediate graphs and use those
results finally to obtain:
%\begin{theorem}\label{1.1}
\vskip 2mm \noindent
\textbf{Theorem \ref{4.15}}. ~
{\it
Let $H$ and $H'$ be two graphs on vertex sets $[p]$ and $[q]$, respectively.
Let $G=H*H'$ be the join of $H$ and $H'$. Let 
$S_{H}=K[x_{i},y_{i} : i\in V(H)]$, $S_{H^{'}}=K[w_{i},z_{i} : i\in V(H^{'})]$ and 
 $S=K[x_{i},y_{i},w_{j},z_{j} : i\in V(H), j\in V(H^{'})]$.
Then}
$$ \Hilb_{S/J_{G}}(t) =\Hilb_{S_{H}/J_{H}}(t)+\Hilb_{S_{H^{'}}/J_{H^{'}}}(t) +
\frac{(p+q-1)t+1}{(1-t)^{p+q+1}}-\frac{(p-1)t+1}{(1-t)^{p+1}}-
\frac{(q-1)t+1}{(1-t)^{q+1}}.$$
%\end{theorem}

As consequences, we obtain the Hilbert series of binomial 
edge ideal of complete $k$-partite graph, wheel graph, fan graph and 
multi-fan graph.

\vskip 2mm
\noindent
\textbf{Acknowledgements:} The authors are grateful to their advisor A. V. Jayanthan for
constant support, valuable ideas and suggestions. The first author thanks the National Board
for Higher Mathematics, India for the financial support. The second author thanks 
University Grant Commission, Government of India for the financial support. 

\section{Preliminaries}
In this section we  recall  some notation and fundamental results which are used
throughout this article.

Let $G$  be a  finite simple graph with the vertex set $V(G)$ and edge set
$E(G)$. For a subset $A \subseteq V(G)$, $G[A]$ denotes the induced
subgraph of $G$ on the vertex set $A$, that is, for $i, j \in A$, 
$\{i,j\} \in E(G[A])$ if and only if $ \{i,j\} \in E(G)$. A subset $U \subseteq V(G)$ 
is said to be a \textit{clique} if $G[U]$ is a complete graph.
A clique $U$ is said to be a maximal clique if for every $v \in V(G) \setminus U$,
$U \cup \{v\}$ is not a clique. A vertex $v \in V(G)$ is called a \textit{free vertex} if
$v$ belongs to exactly one maximal clique.
For a vertex $v$, $G \setminus v$ denotes the  induced subgraph of $G$
on the vertex set $V(G) \setminus \{v\}$. A vertex $v \in V(G)$ is
said to be a \textit{cut} vertex if $G \setminus v$ has strictly more
connected components than $G$.
For a vertex $v$, $G_v$ denotes the graph on the vertex set $V(G)$ and edge set
$E(G_v) =E(G) \cup \{ \{u,w\}: u,w \in N_G(v)\}$, where for any $x\in V(G)$,
$N_G(x) = \{u \in V(G) : \{u,x\} \in E(G) \}$.
Complement of $G$, denoted by $G^c$ is the graph on the vertex
set $V(G)$ and the edge set $E(G^c)= \{\{i,j\} : i \neq j, \{i,j\} \notin E(G) \}$.

For a subset $T$ of $[n]$, let $\bar{T} = [n]\setminus T$ and $c_G(T)$
denote the number of connected components of $G[\bar{T}]$. Let $G_1,\cdots,G_{c_{G}(T)}$ be connected 
components of $G[\bar{T}]$. For each $i$, let $\tilde{G_i}$ denote the complete graph on $V(G_i)$ and
$$P_T(G) = (\underset{i\in T} \cup \{x_i,y_i\}, J_{\tilde{G_1}},\cdots, J_{\tilde{G}_{c_G(T)}}).$$ 
It was shown by Herzog et al. that $J_G =  \underset{T \subseteq [n]}\cap P_T(G)$, \cite{HH1}.
For each $i \in T$, if $i$ is a cut vertex of the graph $G[\bar{T} \cup \{i\}]$,
then we say that $T$ has the cut point property. Let $\mathscr{C}(G) =\{\phi \}
\cup \{ T: T \; \text{has cut point property} \}$ and $\mathscr{M}(G) = \{ P_T(G) : T \in \mathscr{C}(G)\}$. 
In  \cite{HH1}, the authors proved that $P$ is the minimal prime associated to $J_G$ if and only if $P \in \mathscr{M}(G)$.  
 
Let $M =\underset{k \in \mathbb{N}} \bigoplus M_k$ be a graded $S$-module such that for each $k \in \mathbb{N}$, $l(M_k) < \infty$.
The  function $H_M : \mathbb{N}  \rightarrow \mathbb{N}$ defined as $H_M(k) = l(M_k)$ is called the Hilbert function of the module $M$.
The Hilbert series of $M$ is the generating function of the Hilbert function $H_M$ and is denoted by,
$\Hilb_{M}(t)= \underset{k \in \mathbb{N}} \sum l(M_k) t^k$. Let $M$ be a finitely generated graded $S$-module of dimension $d$. 
The Hilbert polynomial of $M$, denoted by $P_M(X)$ is the unique
polynomial with rational coefficients such that $H_M(k) = P_M(k)$ for $k\gg0$. It is known
that one can express the polynomial $P_M(X)$ as $$P_M(X) = \sum_{i=0}^{d-1} (-1)^{d-1-i} e_{d-1-i}(M) \binom{X+i} {i},$$
where $e_j(M)$'s are integers in \cite[Lemma 4.1.4]{bh}. The coefficient $e(M)=e_0(M)$ is called the multiplicity of $M$.

\section{Hilbert series of binomial edge ideal of decomposable graphs}

In this section, we compute the Hilbert series of decomposable graphs in terms of the 
Hilbert series of its indecomposable components.
Throughout this section, by writing $G = G_1 \cup \cdots \cup G_r$, we mean that $G_i$'s are induced subgraphs of G such that $V(G_i) \cap V(G_j) = \phi$ or $\{v_{i,j}\}$
where $v_{i,j}$ is a free vertex in $G_i$ and $G_j$. We begin by recalling the Betti polynomial of $M$.

\begin{definition}
Let $M$ be a finite graded $S$-module. Then we denote by
$$ B_{M}(s,t)=\displaystyle \sum _{i,j} \beta_{ij}(M)s^{i}t^{j}, $$ the Betti polynomial of $M$, where $\beta_{ij}(M)$ are the graded Betti numbers of $M$.
\end{definition}
Recently J. Herzog and G. Rinaldo \cite[Proposition 3]{her2} proved that if
$G=G_{1}\cup G_{2}$, then 
$$ B_{S/J_{G}}(s,t)= B_{S/J_{G_{1}}}(s,t) B_{S/J_{G_{2}}}(s,t)= B_{S_1/J_{G_1}}(s,t)B_{S_2/J_{G_2}}(s,t), $$ 
where $S_i = K[x_j,y_j:j \in V(G_i)]$ for $i \in \{1,2\}$.
   
 By \cite[Lemma 4.1.13]{bh}, the Hilbert series of a graded $S$-module $S/J_G$ is 
 $$\Hilb_{S/J_{G}}(t)=\frac{B_{S/J_{G}}(-1,t)}{(1-t)^{2n}}.$$
 
 As an immediate consequence, we obtain
\begin{theorem}\label{3.3}
Let $G=G_{1}\cup G_{2}$ be a graph on the vertex set $[n]$. Then 
$$ \Hilb_{S/J_{G}}(t)=(1-t)^{2} \Hilb_{S_{1}/J_{G_{1}}}(t) \Hilb_{S_{2}/J_{G_{2}}}(t),$$ 
where $S_{i}=K[{x_{j},y_{j}:j \in V(G_{i})}]$, for $i=1,2$. In particular,
\[\dim(S/J_G) = \dim(S_1/J_{G_1}) +\dim(S_2/J_{G_2})-2 \ \textit{and} \ e(S/J_{G})=e(S_{1}/J_{G_{1}})e(S_{2}/J_{G_{2}}).\]
\end{theorem}
\begin{proof}
For $i=1,2$, let $G_i$ be a graph on vertex set $[m_i]$. Note that $n=m_1+m_2-1$.
Now, \begin{eqnarray*}
\Hilb_{S/J_{G}}(t)&=&\frac{B_{S/J_{G}}(-1,t)}{(1-t)^{2n}} 
\\&=& \frac{B_{S/J_{G}}(-1,t)}{(1-t)^{2(m_{1}+m_{2}-1)}} 
\\ &= &(1-t)^{2} \frac{B_{S_{1}/J_{G_{1}}}(-1,t)}{(1-t)^{2 m_{1}}} \frac{B_{S_{2}/J_{G_{2}}}(-1,t)}{(1-t)^{2 m_{2}}}
\\ &= &(1-t)^{2} \Hilb_{S_{1}/J_{G_{1}}}(t) \Hilb_{S_{2}/J_{G_{2}}}(t).
\end{eqnarray*}

Now, let $d= \dim(S/J_G) $ and $d_{i}$ = $\dim(S_{i}/J_{G_{i}})$, for $i=1,2$.
It follows from \cite[Corollary 4.1.8]{bh} that $ \Hilb_{S/J_{G}}(t) = \frac{Q_{S/J_{G}}(t)}{(1-t)^{d}}$, 
$ \Hilb_{S_1/J_{G_1}}(t) = \frac{Q_{S_1/J_{G_1}}(t)}{(1-t)^{d_1}}$ and 
$ \Hilb_{S_2/J_{G_2}}(t) = \frac{Q_{S_2/J_{G_2}}(t)}{(1-t)^{d_2}}$. 
Therefore \begin{eqnarray*}
\frac{Q_{S/J_{G}}(t)}{(1-t)^{d}} 
&=&(1-t)^{2} \frac{Q_{S_{1}/J_{G_{1}}}(t)}{(1-t)^{d_{1}}}\frac{Q_{S_{2}/J_{G_{2}}}(t)}{(1-t)^{d_{2}}}.
\end{eqnarray*}
%Thus, $$ \frac{Q_{S/J_{G}}(t)}{(1-t)^d} = \frac{Q_{S_{1}/J_{G_{1}}}(t)Q_{S_{2}/J_{G_{2}}}(t)}{(1-t)^{d_{1}+d_2-2}}. $$
Hence $d=d_1+d_2-2$ and $Q_{S/J_{G}}(t) = Q_{S_{1}/J_{G_{1}}}(t)Q_{S_{2}/J_{G_{2}}}(t)$.
Also, $e(S/J_{G}) = Q_{S/J_{G}}(1) =e(S_{1}/J_{G_{1}})e(S_{2}/J_{G_{2}})$.
\end{proof}
 The following is an immediate consequence of previous result.
\begin{corollary}\label{3.5}
Let $G= G_1 \cup \cdots \cup G_r$ be a connected graph. Let $S_i = K[x_j,y_j : j \in V(G_i)]$, for each $i \in [r]$.
Then $$\Hilb_{S/J_G}(t) = (1-t)^{2r-2} \underset{i\in[r]} \prod \Hilb_{S_i/J_{G_i}}(t).$$ In particular,
$\dim(S/J_G)=\displaystyle \sum_{i\in [r]} \dim(S_i/J_{G_i})-2r+2$ and
$e(S/J_G) = \underset{i \in[r]} \prod e(S_i/J_{G_i})$.
\end{corollary}

Note that if $G$ is the complete graph on $n$ vertices, then $S/J_G$ is a determinantal ring. It follows from \cite[Corollary 1]{ch} 
that $\Hilb_{S/J_G}(t) = \frac{(n-1)t+1}{ (1-t)^{n+1}}$. Hence by \cite[Proposition 4.1.9]{bh},
$e(S/J_G) = n$, $e_1(S/J_G) = n-1$ and for $2\leq i \leq n$, $e_i(S/J_G) =0$.
\begin{corollary}\label{3.6}
Let $G=P_n$ be the path graph on the vertex set $[n]$. Then $$\Hilb_{S/{J_G}}(t)= \frac{(1+t)^{n-1}}{(1-t)^{n+1}},$$
 $e_i(S/{J_G}) = \binom{n-1}{i}  2^{n-1-i} $ for $0 \leq i \leq n-1$ and $e_n(S/{J_G}) =0$.
\end{corollary}
\begin{proof}
The result follows for $n=1,2$. For $n \geq 3$, $G$ is a decomposable graph with $n-1$ indecomposable subgraphs 
and each indecomposable subgraph of $G$ is $K_2$. Thus, by Corollary \ref{3.5},
$\Hilb_{S/J_G}(t) = (1-t)^{2n-4} \Big(\frac{t+1}{(1-t)^3}\Big)^{n-1}= \frac{(t+1)^{n-1}}{(1-t)^{n+1}}= \frac{Q(t)}{(1-t)^{n+1}}$.
For $0\leq i \leq n-1$, $Q^{(i)}(t) = {i !}\binom{n-1}{i}(1+t)^{n-1-i}$, where $Q^{(i)}(t)$ denotes the $i$-th derivative of $Q(t)$ with respect to $t$. Hence it follows from 
\cite[Proposition 4.1.9]{bh} that for $0 \leq i \leq n-1$, $e_i(S/J_G)= \binom{n-1}{i}2^{n-1-i}$ and $e_n(S/J_G) =0$.
\end{proof}

It follows from the previous Corollary that the Hilbert Polynomial of $P_n$ is
$P_{S/J_{P_n}}(X) = \underset{i \in [n]} \sum (-1)^{n-i} 2^{i-1} \binom{n-1}{i-1}\binom{X+i}{i}$.

Now, we obtain the Hilbert series of $k$-handle lollipop graph.
Let $m \geq 2$ and $K_m$ be the complete graph on $[m]$. Let $L_{m,r_1,\ldots ,r_k}$ denote
the graph obtained by identifying one free vertex of each path $P_{r_1}, \ldots ,P_{r_k}$ with 
$k$ distinct vertices of $K_m$. Such a graph is called $k$-handle lollipop graph.

 For an induced subgraph $H$ of $G$, set 
 $S_H = K[x_j,y_j :j \in V(H)]$.
\begin{proposition}\label{3.7}
Let $G=L_{m,r_{1},...,r_{k}}$ be a $k$-handle lollipop graph. Then 
$$ \Hilb_{S/J_{G}}(t)=\frac{((m-1)t+1)(1+t)^{r-k}}{(1-t)^{m+r-k+1}}, \ where \ r=r_{1}+\cdots +r_{k}. $$ 
In particular, $\dim(S/J_G)= m+r-k+1$ and  $e(S/J_{G})=2^{r-k}m$.
\end{proposition}
 \begin{proof}
 Note that $ G=K_m\cup P_{r_{1}}\cup P_{r_{2}} \cdots \cup P_{r_{k}}$. Therefore by Corollary \ref{3.5},
 $$ \Hilb_{S/J_{G}}(t)=(1-t)^{2k}\Hilb_{S_{K_m}/J_{K_m}}(t) \underset{i \in [k]}\prod\Hilb_{S_{P_{r_i}}/J_{P_{r_i}}}(t).$$
 By Corollary \ref{3.6}, we have 
 \begin{eqnarray*}
 \Hilb_{S/J_{G}}(t) &= &(1-t)^{2k} \frac{((m-1)t+1)}{(1-t)^{m+1}}\underset {i \in [k]} \prod \frac{(1+t)^{r_i -1}}{(1-t)^{r_i+1}}
 \\&=& \frac{((m-1)t+1)(1+t)^{r-k}}{(1-t)^{m+r-k+1}}.
 \end{eqnarray*}
 \end{proof}
 
 It follows from Proposition \ref{3.7} and   \cite[Proposition 4.1.9]{bh} that if
 $G=L_{m,r_1,\ldots, r_k}$, then for $1\leq i \leq r-k$, $e_i(S/J_G) = (m-1)\binom{r-k}{ i-1}2^{r-k-i+1} + m\binom{r-k}{ i}2^{r-k-i}$,
 $e_{r-k+1}(S/J_G) = m-1$ and $e_j(S/J_G) =0$ for $j \geq r-k+2$.

Let $G$ be a  Cohen-Macaulay closed graph on vertex set  $[n]$. It follows from \cite[Theorem 3.1]{her1} 
that $G=K_{m_1} \cup \cdots \cup K_{m_r}$ for some $m_i \geq 2$.
Therefore by using Corollary \ref{3.5}, we recover the result of Ene-Herzog-Hibi(\cite[page 66]{her1}), \[\Hilb_{S/J_G}(t) = (1-t)^{2r-2}\underset{i \in [r]} \prod 
\Hilb_{S_{K_{m_i}}/J_{K_{m_i}}}(t) = \frac{\underset{i \in[r]} \prod ((m_i-1)t+1)}{(1-t)^{n+1}}\] and $e(S/J_G) = \underset{i \in [r]} \prod m_i$.

\section{Hilbert series of Binomial edge ideal of Join of Graphs }
In this section we compute the Hilbert series of binomial edge ideal of join of two graphs. 
As  consequences of this we obtain the Hilbert series of the binomial edge
ideal of complete $k$-partite graphs, wheel graphs, fan graphs and multi-fan graphs.

\begin{definition}
 Let $H$ and $H'$ be two graphs  with the vertex sets $[p]$ and $[q]$, respectively. The \textit{join} of $H$ and $H'$, denoted by $H*H'$ 
 is the graph with vertex set $[p] \sqcup [q]$ and the edge set $E(H*H')= E(H) \cup E(H') \cup \{\{i,j\}| i \in [p], j \in [q]\}$.
\end{definition}

\noindent
\begin{minipage}{\linewidth}
\begin{minipage}{.45\linewidth}
The graph $G$, given on the right, is the join of the complete graph $K_4$ and the complement graph of the complete graph $K_3$. 
\end{minipage}
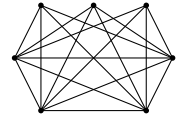
\begin{minipage}{.5\linewidth}
\captionsetup[figure]{labelformat=empty}
\begin{figure}[H]
\begin{tikzpicture}[rotate=90, scale=.70]
%\clip(-4.3,-3.2) rectangle (7.4,6.3);
\draw (4,2)-- (2,2);
\draw (2,2)-- (4,3);
\draw (4,3)-- (2,4);
\draw (2,4)-- (4,2);
\draw (4,4)-- (2,2);
\draw (2,4)-- (4,4);
\draw (2,4)-- (3,4.5);
\draw (3,4.5)-- (3,1.5);
\draw (3,1.5)-- (2,2);
\draw (2,4)-- (2,2);
\draw (2,2)-- (3,4.5);
\draw (2,4)-- (3,1.5);
\draw (4,4)-- (3,4.5);
\draw (4,4)-- (3,1.5);
\draw (4,3)-- (3,1.5);
\draw (4,3)-- (3,4.5);
\draw (4,2)-- (3,1.5);
\draw (4,2)-- (3,4.5);
\begin{scriptsize}
\fill (2,4) circle (1.5pt);
\fill  (3,4.5) circle (1.5pt);
\fill  (3,1.5) circle (1.5pt);
\fill  (2,2) circle (1.5pt);
\fill (4,4) circle (1.5pt);
%\draw (4.42,3.98) node {$v_1$};
\fill (4,3) circle (1.5pt);
%\draw(4.4,3.08) node {$v_2$};
\fill (4,2) circle (1.5pt);
%\draw (4.42,2.06) node {$v_3$};
\end{scriptsize}
\end{tikzpicture}
\caption{$G=K_4*K_3^c$}
\end{figure}
\end{minipage}
\end{minipage}

For rest of the section we use the following notation: Let $H$ and $H'$ be  two graphs on vertex sets $[p]$ and $[q]$, respectively.
Let $G=H*H^{'}$ be the join of $H$ and $H'$. Set $S_{H}=K[x_{i},y_{i}: i\in V(H)]$, $
S_{H^{'}}=K[z_{j},z'_{j}: j\in V(H^{'})]$ and $S=K[x_{i},y_{i},z_{j},z'_{j}: i\in V(H), j\in V(H^{'})]$.

 \begin{theorem}\label{4.3}
 Let $H$ and $H^{'}$ be two disconnected graphs on 
 vertex sets $[p]$ and $[q]$, respectively.
 Let $G=H*H^{'}$ be the join of $H$ and $H^{'}$. Then  
$$ \Hilb_{S/J_{G}}(t) =\Hilb_{S_{H}/J_{H}}(t)+\Hilb_{S_{H^{'}}/J_{H^{'}}}(t) +
\frac{(p+q-1)t+1}{(1-t)^{p+q+1}}-\frac{(p-1)t+1}{(1-t)^{p+1}}-\frac{(q-1)t+1}{(1-t)^{q+1}}.$$
 \end{theorem}
 \begin{proof}
 It follows from \cite[Proposition 4.14]{KM6} that \[ \mathscr{C}(G)=\{\phi\} \cup \{T\sqcup[q]:T \in \mathscr{C}(H)\} \cup\{[p] \sqcup T': T' \in \mathscr{C}(H')\}. \]
Let  $$Q_{1}=\underset{{\substack{T\in \mathscr{C}(G)\\ [p] \subseteq T}}}{\cap} P_{T}(G)=(x_{i},y_{i} : i\in [p])+J_{H^{'}}$$ and 
 $$ Q_{2} = \underset{{\substack{T\in \mathscr{C}(G) \\ T\neq \phi \ [p] \nsubseteq T}}} {\cap}P_{T}(G)
  =(z_{i},w_{i} : i\in [q])+J_{H}.$$
Now, by \cite[Corollary 3.9]{HH1},
$$ J_G= \underset{T \in \mathscr{C}(G)} \cap P_T(G) = P_{\phi}(G) \cap Q_1 \cap Q_2 = J_{K_{p+q}} \cap Q_1 \cap Q_2.$$
Set $Q_3 = J_{K_{p+q}} \cap Q_2$ and consider the short exact sequence,
 $$ 0\longrightarrow \frac{S}{Q_{3}} \longrightarrow \frac{S}{J_{K_{p+q}}} \oplus \frac{S}{Q_{2}} \longrightarrow \frac{S}{J_{K_{p+q}} +Q_2} \longrightarrow 0 .$$
Note that  $J_{K_{p+q}}+Q_2= (z_{i},w_{i} : i\in [q])+J_{K_{p}}.$
Therefore 
 \begin{align*}
 \Hilb_{S/Q_{3}}(t) &=\Hilb_{S/J_{K_{p+q}}}(t) + \Hilb_{S/Q_2}(t) - \Hilb_{S/({J_{K_p+q}+Q_2})}(t)
 \\ &=\frac{(p+q-1)t+1}{(1-t)^{p+q+1}} + \Hilb_{S_{H}/J_{H}}(t) - \frac{(p-1)t+1}{(1-t)^{p+1}}. 
 \end{align*}
 Now, consider the short exact sequence
 $$ 0\longrightarrow \frac{S}{J_{G}}\longrightarrow \frac{S}{Q_{1}}\oplus \frac{S}{Q_{3}}\longrightarrow \frac{S}{Q_{1}+Q_{3}} \longrightarrow 0 .$$
Note that  $Q_{1}+Q_{3} =(x_{i},y_{i} : i\in [p])+ J_{K_{q}}$. Therefore
 \begin{align*}
 \Hilb_{S/J_{G}}(t) &= \Hilb_{S/Q_{1}}(t)+ \Hilb_{S/Q_{3}}(t)- \Hilb_{S/Q_{1}+Q_{3}}(t) 
 \\ &= \Hilb_{S_{H^{'}}/J_{H^{'}}}(t)+\Hilb_{S_{H}/J_{H}}(t)+ \\  & 
 \ \ \ \ \ \ \frac{(p+q-1)t+1}{(1-t)^{p+q+1}} - \frac{(p-1)t+1}{(1-t)^{p+1}}-\frac{(q-1)t+1}{(1-t)^{q+1}}.
 \end{align*}
 
 \end{proof}
Now we move on to study the join of a disconnected graph with a complete graph. We first identify the subsets with cutpoint property.
 \begin{lemma}\label{4.5}
 Let $H$ be a disconnected graph on the vertex set $[p]$.
 Let $G=H*K_q$ be the join of $H$ and $K_q$. Then $$\mathscr{C}(G)= \{\phi \} \cup \{T\sqcup [q]: T\in \mathscr{C}(H) \}.$$
 \end{lemma}
\begin{proof}
 The proof is immediate from \cite[Proposition 4.5]{KM6}.
\end{proof}
We now compute the Hilbert series:
 \begin{theorem}\label{4.6} 
 Let $H$ be a disconnected graph on the vertex set $[p]$.
 Let $G=H*K_{q}$ be the join of $H$ and the complete graph $K_q$. Then 
 $$ \Hilb_{S/J_{G}}(t)=\Hilb_{S_{H}/J_{H}}(t)+ \frac{(p+q-1)t+1}{(1-t)^{p+q+1}}  -\frac{(p-1)t+1}{(1-t)^{p+1}}. $$
 \end{theorem}
 \begin{proof}
 By \cite[Corollary 3.9]{HH1}, $J_G = \underset{T \in \mathscr{C}(G)} \cap P_T(G)$.
 Let $Q = \underset{{\substack{T\in \mathscr{C}(G)\\ [q] \subseteq T}}} \cap P_{T}(G)=(z_{j},w_{j}: j\in [q])+J_{H}$.
 Now it follows from Lemma \ref{4.5} that $J_G = P_{\phi}(G) \cap Q = J_{K_{p+q}} \cap Q$.
 Consider  the short exact sequence
$$ 0\longrightarrow \frac{S}{J_{G}}\longrightarrow \frac{S}{J_{K_{p+q}}}\oplus \frac{S}{Q}\longrightarrow \frac{S}{J_{K_{p+q}} +Q} \longrightarrow 0.$$
Note that $J_{K_{p+q}} +Q = (z_j,w_j : j \in [q] ) + J_{K_p}$. Therefore 
\begin{align*}
\Hilb_{S/J_{G}}(t) &=  \Hilb_{S/Q}(t)+\Hilb_{S/ J_{K_{p+q}}}(t)- \Hilb_{S/({J_{K_{p+q}} +Q})}(t) 
\\ &= \Hilb_{S_{H}/J_{H}}(t)+ \frac{(p+q-1)t+1}{(1-t)^{p+q+1}}- \frac{(p-1)t+1}{(1-t)^{p+1}}.
\end{align*}
\end{proof}

If $H$ is a disconnected graph, then by taking $q=1$ in Theorem \ref{4.6}, we obtain the 
Hilbert series of the cone of $H$. Of particular interest is the case of multi-fan graph. 
Let $P_{p_1},\ldots, P_{p_r}$ denote paths on vertices $p_1, \ldots, p_r$ respectively.
The the graph $\{v\} * (\sqcup_{i=1}^r P_{p_i})$
is called a multi-fan graph, denoted by $M_{p_1, \ldots, p_r}.$
As a consequence of Theorem \ref{4.6}, we obtain the Hilbert series of binomial edge ideal of the
multi-fan graphs. 

\begin{corollary}
Let $G=M_{p_1,\ldots,p_r}$ be  the multi-fan graph with $r \geq 2$.  Let $p=\underset{i\in [r]}\sum p_{i}$. Then
$$\Hilb_{S/J_G} = \frac{(1+t)^{p-r}}{(1-t)^{p+r}} + \frac{(p-1)t^2 +2t}{(1-t)^{p+2}}.$$ In particular,
$ \dim(S/J_{G}) =p+r$ and 
\[
e(S/J_{G}) = 
     \begin{cases}
       {2}^{p-r} &\quad\text{if r}> 2, \\
       {2}^{p-r}+{p+1} &\quad\text{if r} = 2.
     \end{cases}
\]
\end{corollary}

\begin{proof}
Let $H = \underset{ i \in [r]} \sqcup P_{p_i}$. Since $r \geq 2$, by Theorem \ref{4.6} and Corollary \ref{3.6}
 \begin{align*}
\Hilb_{S/J_{G}}(t) &=\Hilb_{S_{H}/J_{H}}(t)+\frac{pt+1}{(1-t)^{p+2}} - \frac{(p-1)t+1}{(1-t)^{p+1}}
\\ &= \prod_{i=1}^r \Hilb_{S_{P_{p_{i}}}/J_{P_{p_{i}}}}(t) +\frac{(p-1)t^2+2t}{(1-t)^{p+2}}
\\ &= \prod_{i=1}^{r} \frac{(1+t)^{p_{i}-1}}{(1-t)^{p_{i}+1}} +\frac{(p-1)t^2+2t}{(1-t)^{p+2}}
\\ &= \frac{(1+t)^{p-r}}{(1-t)^{p+r}}+\frac{(p-1)t^2+2t}{(1-t)^{p+2}}. 
\end{align*}
\end{proof} 
We proceed to study the connected case now. To understand $H*K_q$, where
$H$ is a connected graph, we need more tools. We first treat the case $q=1$.
\begin{theorem}\label{4.9}
Let $H$ be a connected graph on $[p]$ and $G=H*\{v\}$. Then 
 $$ \Hilb_{S/J_{G}}(t) =\Hilb_{S_{H}/J_{H}}(t)+ \frac{(p-1)t^2+2t}{(1-t)^{p+2}} .$$
 In particular,\[   
e(S/J_{G}) = 
     \begin{cases}
     {e(S}_{H}/J_{H}) &\quad\text{if dim(S}_{H}/J_{H})\geq p+3, \\
       {p+1}+{e(S}_{H}/J_{H}) &\quad\text{if dim(S}_{H}/J_{H})= p+2, \\
       {p+1} &\quad\text{if dim(S}_\text{H}/J_\text{H})=p+1.
     \end{cases}
\]
\end{theorem}
\begin{proof}
 If $H=K_p$, then the result is immediate. Assume that $H \neq K_p$, 
 so that $v$ is not a free vertex in $G$. By \cite[Lemma 4.8]{oh}, 
 $J_G = J_{G_v} \cap ((x_v,y_v) + J_{G \setminus v})$. Note that 
 $G_v = K_{p+1}$ and ${G \setminus v}= H$. It follows from the short exact sequence
 $$ 0\longrightarrow \frac{S}{J_{G}}\longrightarrow 
 \frac{S}{J_{K_{p+1}}}\oplus \frac{S}{(x_v,y_v)+J_H}\longrightarrow \frac{S}{(x_v,y_v)+J_{K_{p}}} \longrightarrow 0$$
 that
 \begin{align*}
 \Hilb_{S/J_{G}}(t) &=\Hilb_{S_{H}/J_{H}}(t)+\frac{pt+1}{(1-t)^{p+2}}-\frac{(p-1)t+1}{(1-t)^{p+1}} 
 \\ &=\Hilb_{S_{H}/J_{H}}(t)+\frac{(p-1)t^2+2t}{(1-t)^{p+2}}. 
 \end{align*}
 The formula for the multiplicity follows directly from the Hilbert series expression.
 \end{proof}
 
As an immediate consequence, we have:
 \begin{corollary}
 Let $G=W_{p+1}$ be the wheel graph on $[p+1]$ with $p \geq 3$. Then 
 $$ \Hilb_{S/J_{G}}(t)=\frac{(1-t^2)(1+t)^{p-1}+(p-1)t^{p}+t^{p+1}}{(1-t)^{p+1}}+\frac{(p-1)t^2+2t}{(1-t)^{p+2}}.$$
 In particular, $\dim(S/J_G)=p+2$ and $e(S/J_G)=p+1$.
 \end{corollary} 
\begin{proof}
Let $H = C_p $ be the cycle graph on $[p]$. Note that $G=C_{p}*\{v\}$. Therefore by  Theorem \ref{4.9},
$$ \Hilb_{S/J_{G}}(t)=\Hilb_{S_{H}/J_{H}}(t)+\frac{(p-1)t^2+2t}{(1-t)^{p+2}}.$$
Now the result follows from \cite[Theorem 10]{SZ}.
\end{proof}

We now define a new product which is required in the computation of the
Hilbert series of join of two arbitrary graphs.

%\begin{minipage}{\linewidth}
\begin{minipage}{0.5\linewidth}
Let $H$ be a graph on $[p]$ and $H'$ be a graph on $[q]$. 
Let $v_{1},...,v_{r}$ be vertices of $H'$ with $1 \leq r \leq q$. The graph with vertex set $[p]\sqcup [q]$ 
and edge set \[E(H)\sqcup \left\{ (u,v_{i}):u \in V(H),i=1,...,r \right\}\sqcup E(H')\] is denoted by $H (*)^{r} H'$. 
Note that if $r= q$, then $H(*)^q H'$ is the join of $H$ and $H'$. For example, the graph $G$ given on the right side is the graph $P_3(*)^2K_4$.
\end{minipage}
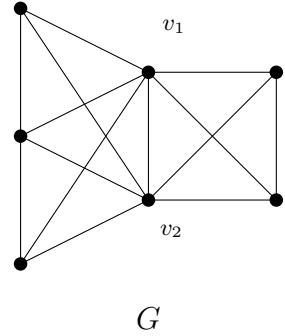
\begin{minipage}{.45\linewidth}
\captionsetup[figure]{labelformat=empty}
\begin{figure}[H]
\begin{tikzpicture}[scale=1.7]
%\clip(-4.3,-3.2) rectangle (7.4,6.3);
\draw (1,4.5)-- (1,3.5);
\draw (1,3.5)-- (1,2.5);
\draw (1,2.5)-- (2,3);
\draw (2,3)-- (1,3.5);
\draw (1,3.5)-- (2,4);
\draw (2,4)-- (1,2.5);
\draw (1,4.5)-- (2,3);
\draw (2,4)-- (1,4.5);
\draw (2,4)-- (3,4);
\draw (3,4)-- (3,3);
\draw (3,3)-- (2,3);
\draw (2,4)-- (2,3);
\draw (2,3)-- (3,4);
\draw (2,4)-- (3,3);
\begin{scriptsize}
\fill (2,4) circle (1.5pt);
\draw(2.2,4.36) node {$v_1$};
\fill(3,4) circle (1.5pt);
\fill (3,3) circle (1.5pt);
\fill  (2,3) circle (1.5pt);
\draw(2.18,2.76) node {$v_2$};
\fill(1,4.5) circle (1.5pt);
\fill(1,3.5) circle (1.5pt);
\fill (1,2.5) circle (1.5pt);
\end{scriptsize}
\end{tikzpicture}
\caption{$ G $}
\end{figure}
\end{minipage}
%\end{minipage}

Our next aim is to compute the Hilbert series of $H(*)^r K_q.$ First we identify the sets with cut point property.

\begin{lemma}\label{4.11}
Let $H$ be a graph on $[p]$ and $G=H (*)^{r} K_{q},$ where $q\geq 2$ and $1 \leq r <q$. Then
\[ \mathscr{C}(G)=\{\phi \} \cup \{ T\sqcup \{v_{1},...,v_{r} \} : T\in \mathscr{C}(H) \}.\]
\end{lemma}
\begin{proof}
Let $T\in \mathscr{C}(G)$ and $T\neq \phi$. First note that $\{v_1,...,v_r\} \subseteq T$. For if $v_{i}\notin T$ for some $i$, then $G_{\bar{T}}$ is a connected graph, 
which contradicts the fact that $T\in \mathscr{C}(G)$. Let $T'=T \setminus \{v_{1},...,v_{r}\}$. 
We need to show that $T^{'}\in \mathscr{C}(H)$. If $T'=\phi$, then we are through. Assume that $T^{'}\neq \phi$. If
$u\in T^{'}$, then  $G[{\bar{T} \cup \{ u \}}]=H[{\bar{T'}\cup \{ u \}}] \sqcup K_{q-r}$. Since $u$ is a cut vertex of 
$G[{\bar{T} \cup \{ u \}}]$, it is a cut vertex of $H[{\bar{T'}\cup \{ u \}}]$. Thus $T' \in \mathscr{C}(H)$.
 
Conversely, let $T=T' \sqcup \{v_{1},...,v_{r} \}$, where $T'\in \mathscr{C}(H)$. 
Note that for $ i \in [r]$, $G[{\bar{T}\cup \{ v_{i} \}}]=H[{\bar{T'}}] (*)^{1} K_{q-r+1}$. So $v_{i}$ is a cut vertex of $ G[{\bar{T}\cup \{ v_{i} \}}]$.
If $T' = \phi$, then we are done. Assume that $T' \neq \phi$ and let $ u \in T'$. Since
$G[{\bar{T}\cup \{ u \}}]=H[{\bar{T'}\cup \{ u \}}] \sqcup K_{q-r}$, $u$ is a cut vertex of $G[{\bar{T}\cup \{ u \}}]$ and $T\in \mathscr{C}(G)$.
Hence the result follows.
\end{proof}
\begin{theorem}\label{4.12}
Let $H$ be a graph on $[p]$ and $G=H (*)^{r} K_{q}$ be the graph on vertex set $[p]\sqcup [q]$ with $q\geq 2$ and $1 \leq r <q$. Then
$$ \Hilb_{S/J_{G}}(t) =\Hilb_{S_{H}/J_{H}}(t) \frac{(q-r-1)t+1}{(1-t)^{q-r+1}}+\frac{(p+q-1)t+1}{(1-t)^{p+q+1}}-\frac{(p+q-r-1)t+1}{(1-t)^{p+q-r+1}}.$$
\end{theorem}
\begin{proof}
By \cite[Corollary 3.9]{HH1}, $J_G = \underset{ T \in \mathscr{C}(G)} \cap P_T(G)$.
Now, by Lemma \ref{4.11},
\begin{align*}
J_{G} &=P_{\phi}(G)\cap \underset{T \in \mathscr{C}(G) , T\neq \phi} {\cap}  P_{T}(G)
\\ &= J_{K_{p+q}} \cap ( (x_{v_{i}},y_{v_{i}}: i \in [r]) +J_{H}+J_{K_{q-r}}).
\end{align*}
Let $Q=(x_{v_{i}},y_{v_{i}}: i \in [r])+J_{H}+J_{K_{q-r}}$. Note  that 
$J_{K_{p+q}}+Q=(x_{v_{i}},y_{v_{i}}: i \in [r])+J_{K_{p+q-r}}.$ Then it follows from the exact sequence below
$$ 0\longrightarrow \frac{S}{J_{G}}\longrightarrow \frac{S}{J_{K_{p+q}}}\oplus \frac{S}{Q}\longrightarrow \frac{S}{J_{K_{p+q}}+Q}\longrightarrow 0 $$
 that
\begin{align*}
\Hilb_{S/J_{G}}(t) &= \Hilb_{S/Q}(t)+\Hilb_{S/J_{K_{p+q}}}(t) - \Hilb_{S/(J_{K_{p+q}}+Q)}(t) 
\\ &=\Hilb_{S_{H}/J_{H}}(t) \frac{(q-r-1)t+1}{(1-t)^{q-r+1}}+\frac{(p+q-1)t+1}{(1-t)^{p+q+1}}-\frac{(p+q-r-1)t+1}{(1-t)^{p+q-r+1}}.
\end{align*}
\end{proof}
Now we compute the Hilbert series of a q-cone over a connected graph $H$, i.e., join of $H$ and $q$ isolated
vertices. We  denote the graph of q isolated vertices by $K_q^c$.

\begin{theorem}\label{4.13}
Let $H$ be a connected graph on the vertex set $[p]$. Let $G=H*K_{q}^{c}$ be the join of $H$ and $K_{q}^{c}$. Then 
$$ \Hilb_{S/J_{G}}(t)=\Hilb_{S_{H}/J_{H}}(t) + \Hilb_{S/J_{K_{p}*{K_{q}^c}}}(t) - \frac{(p-1)t+1}{(1-t)^{p+1}}.$$
\end{theorem}
\begin{proof}
We proceed by induction on $q$. For $q=1$, $G=H*\{v\}$ and the result follows from Theorem  \ref{4.9}.
Now assume that $q> 1$ and the result is true for $q-1$. Let $G=H*K_{q}^{c}$ and 
$V(K_{q}^{c})=\left\{v_{1},v_{2},...,v_{q} \right\}.$ Let $Q_{1}=(x_{v_{q}},y_{v_{q}})+J_{G\setminus v_{q}}$ 
and $Q_{2}=J_{G_{v_{q}}}$. Note that $Q_{1}+Q_{2}=(x_{v_{q}},y_{v_{q}})+J_{G_{v_{q}}\setminus v_{q}}$.
Since $v_{q}$ is not a free vertex, it follows from \cite[Lemma 4.8]{oh} that $J_{G}=Q_{1}\cap Q_{2}$. Let $R=K[x_{i},y_{i} : i\in V(G) \ and \ i\neq v_{q}]$. Then by induction hypothesis 
$$ \Hilb_{S/{Q_{1}}}(t)=\Hilb_{R/J_{G\setminus v_{q}}}(t)=\Hilb_{S_{H}/J_{H}}(t)+ \Hilb_{R/J_{K_{p}*K_{q-1}^c}}(t) -\frac{(p-1)t+1}{(1-t)^{p+1}}.$$ 
Consider the short exact sequence
$$ 0\longrightarrow \frac{S}{J_{G}}\longrightarrow \frac{S}{Q_{1}}\oplus \frac{S}{Q_{2}}\longrightarrow \frac{S}{Q_{1}+Q_{2}} \longrightarrow 0. $$
Note that $ G_{v_q} = K_p*K_{q}^c$, $G_{v_q} \setminus v_q = K_p*K_{q-1}^c$ and $ G \setminus v_q =H*K_{q-1}^c$.
Therefore, it follows from the above short exact sequence that
$$\Hilb_{S/J_{G}}(t) =\Hilb_{S_{H}/J_{H}}(t) + \Hilb_{S/J_{K_{p}*K_{q}^c}}(t) - \frac{(p-1)t+1}{(1-t)^{p+1}}.$$
\end{proof}
As a consequence of the Theorems
\ref{4.6}, \ref{4.9} and \ref{4.13}, we compute the Hilbert series of fan graphs.
\begin{corollary}
Let $G=P_{p}*K_{q}^{c}$ be the fan graph. Then
$$\Hilb_{S/J_{G}}(t) =\frac{(1+t)^{p-1}}{(1-t)^{p+1}}+ \frac{1}{(1-t)^{2q}}+ \frac{(p+q-1)t+1}{(1-t)^{p+q+1}} -
\frac{(p-1)t+1}{(1-t)^{p+1}}- \frac{(q-1)t+1}{(1-t)^{q+1}}.$$
In particular, $\dim(S/J_{G})=\max\{2q,p+q+1 \}$ and 
\[
e(S/J_{G}) = 
     \begin{cases}
      {1} & \quad\text{if p}+1 < q, \\
       {p+q+1} &\quad\text{if p}+1 = q, \\
         {p+q} &\quad\text{if p}+1 > q. 
     \end{cases}
\]
\end{corollary}
\begin{proof}
By Theorem \ref{4.13}, $$\Hilb_{S/J_{G}}(t) = \frac{(1+t)^{p-1}}{(1-t)^{p+1}}+ \Hilb_{S/J_{K_{p}*K_{q}^c}}(t) - \frac{(p-1)t+1}{(1-t)^{p+1}}.$$ 
If $q=1$, then the assertion follows from the Theorem \ref{4.9}.
Assume that $q > 1$. Now by Theorem \ref{4.6}, the assertion follows.
\end{proof}
\begin{theorem}\label{4.14}
Let $H$ be a connected graph on $[p]$ and $G=H*K_{q}$, where $q\geq 2$. Then 
$$ \Hilb_{S/J_{G}}(t)=\Hilb_{S_{H}/J_{H}}(t)+\frac{(p+q-1)t+1}{(1-t)^{p+q+1}} - \frac{(p-1)t+1}{(1-t)^{p+1}}.$$
\end{theorem}
\begin{proof}
Let $H^{'}=H\sqcup \{ w\}$, $G^{'}=H^{'}*K_{q}^{c}$, $ S_{H^{'}}=S_{H}[x_{w},y_{w}]$ and $S^{'}=S[x_{w},y_{w}]$. By Theorem \ref{4.3}, 
$$\Hilb_{S^{'}/J_{G^{'}}}(t) =\frac{1}{(1-t)^2} \Hilb_{S_{H}/J_{H}}(t)+\frac{1}{(1-t)^{2q}}+
\frac{(p+q)t+1}{(1-t)^{p+q+2}}-\frac{pt+1}{(1-t)^{p+2}}-\frac{(q-1)t+1}{(1-t)^{q+1}}.$$   
Since $w$ is not a free vertex of $G^{'}$, it follows from \cite[Lemma 4.8]{oh}
 that $J_{G^{'}}= ((x_{w},y_{w})+J_{G^{'}\setminus w})\cap J_{G^{'}_{w}}$.
 Let $Q_{1}=(x_{w},y_{w})+J_{G^{'} \setminus w}$, $Q_{2}=J_{G^{'}_{w}}$. 
 Note that $Q_{1}+Q_{2}=(x_{w},y_{w})+J_{G^{'}_{w}\setminus w}$,
 $G^{'}_{w}=H (*)^{q} K_{q+1}$, ${G^{'} \setminus w}=H*K_{q}^{c}$ and ${G^{'}_{w} \setminus w}=H*K_{q}=G$. 
 Consider the short exact sequence
$$ 0\longrightarrow \frac{S^{'}}{J_{G^{'}}}\longrightarrow \frac{S^{'}}{Q_{1}}\oplus \frac{S^{'}}{Q_{2}}\longrightarrow \frac{S^{'}}{Q_{1}+Q_{2}} \longrightarrow 0 .$$
Then by Theorem \ref{4.12} and Theorem \ref{4.13}
\begin{align*}  
\Hilb_{S/J_{G}}(t) & =\Hilb_{S/J_{H*K_{q}^{c}}}(t)+\Hilb_{S^{'}/J_{H (*)^{q} K_{q+1}}}(t)-\Hilb_{S^{'}/J_{G^{'}}}(t)
\\ &= \Hilb_{S_{H}/J_{H}}(t)+ \frac{(p+q-1)t+1}{(1-t)^{p+q+1}}- \frac{(p-1)t+1}{(1-t)^{p+1}}.
\end{align*}
\end{proof}
Now we are ready to compute the Hilbert series of the join of two arbitrary graphs.
\begin{theorem}\label{4.15}
Let $H$ and $H'$ be graphs on vertex sets $[p]$ and $[q]$, respectively. Let $G=H*H^{'}$ be the join of $H$ and $H'$. 
Then $$ \Hilb_{S/J_{G}}(t)=\Hilb_{S_{H}/J_{H}}(t)+\Hilb_{S_{H^{'}}/J_{H^{'}}}(t)+
\frac{(p+q-1)t+1}{(1-t)^{p+q+1}}-\frac{(p-1)t+1}{(1-t)^{p+1}}-\frac{(q-1)t+1}{(1-t)^{q+1}} .$$
\end{theorem}
\begin{proof}
 If both $H$ and $H'$ are disconnected graphs, then the assertion is proved in Theorem \ref{4.3}.
Assume that $H$ is connected and $H'$ is disconnected. Let $w$ be a new vertex. Set 
$H^{''}=H\sqcup \left\{ w \right\},$ $G^{'}=H^{''}*H^{'}$, $S_{H^{''}}=S_{H}[x_{w},y_{w}]$ and $S^{'}=S[x_{w},y_{w}]$. 
Let $Q_{1}=(x_{w},y_{w})+J_{G' \setminus w}$ and $Q_{2}=J_{G^{'}_{w}}$. 
Since $w$ is not a free vertex of $G'$, it follow from \cite[Lemma 4.8]{oh} that $J_{G^{'}}=Q_{1}\cap Q_{2}$.
Note that $Q_{1}+Q_{2}=(x_{w},y_{w})+J_{G^{'}_{w}\setminus w}$,
$G^{'}_{w}=H (*)^{q} K_{q+1},$ $G^{'}_{w}\setminus w=H*K_{q}$ and $G' \setminus w= G$.
Therefore, it follows from the short exact sequence
$$ 0\longrightarrow \frac{S^{'}}{J_{G^{'}}}\longrightarrow \frac{S^{'}}{Q_{1}}\oplus \frac{S^{'}}{Q_{2}}\longrightarrow \frac{S^{'}}{Q_{1}+Q_{2}} \longrightarrow 0 $$
that $$ \Hilb_{S/J_{G}}(t)=\Hilb_{S^{'}/J_{G^{'}}}(t)+\Hilb_{S/J_{H*K_q}}(t)-\Hilb_{S^{'}/J_{H (*)^q K_{q+1}}}(t) .$$
Note that $G'$ is join of two disconnected graphs.
Now, the assertion follows from Theorems \ref{4.3}, \ref{4.12} and \ref{4.14}.\\
Assume now that both $H$ and $H'$ are connected.  Let $u$ be a new vertex. Set 
$H^{'''}=H'\sqcup \left\{ u \right\},$ $G^{''}=H*H^{'''}$, $S_{H^{'''}}=S_{H'}[x_{u},y_{u}]$ and $S^{'}=S[x_{u},y_{u}]$.
Let $Q_{1}=(x_{u},y_{u})+J_{G}$ and $Q_{2}=J_{G^{''}_{u}}$.
It follow from the \cite[Lemma 4.8]{oh} that $J_{G^{''}}=Q_{1}\cap Q_{2}$. 
Consider the short exact sequence
$$ 0\longrightarrow \frac{S^{'}}{J_{G^{''}}}\longrightarrow \frac{S^{'}}{Q_{1}}\oplus \frac{S^{'}}{Q_{2}}\longrightarrow \frac{S^{'}}{Q_{1}+Q_{2}} \longrightarrow 0 .$$
Note that $Q_{1}+Q_{2}=(x_{u},y_{u})+J_{G^{''}_{u}\setminus u}$, $G^{''}_{u}=H' (*)^{q} K_{q+1},$ $G^{''}_{u}\setminus u=H'*K_{q}$ and $G'' \setminus u= G$.
Then
$$ \Hilb_{S/J_{G}}(t)= \Hilb_{S^{'}/J_{G^{''}}}(t)+ \Hilb_{S/J_{H'*K_q}}(t)-\Hilb_{S^{'}/J_{H' (*)^q K_{q+1}}}(t) .$$
Since $G''$ is join of a connected graph $H$ and a disconnected graph $H'''$, by previous case 
$$ \Hilb_{S'/J_{G''}}(t)=\Hilb_{S_{H}/J_{H}}(t)+\Hilb_{S_{H^{'''}}/J_{H^{'''}}}(t)+
\frac{(p+q)t+1}{(1-t)^{p+q+2}}-\frac{(p-1)t+1}{(1-t)^{p+1}}-\frac{qt+1}{(1-t)^{q+2}} .$$
Now, the assertion follows from Theorems \ref{4.12} and \ref{4.14}.
\end{proof}

As an immediate consequence, we are able to compute the Hilbert series, dimension and multiplicity of complete multi-partite graphs. 
\begin{corollary}
 Let $G=K_{p_1,\ldots,p_k}$ be the complete $k$-partite graph on the vertex set $[n] = [p_1] \sqcup \cdots \sqcup [p_k]$ with $p_1 \geq \cdots \geq p_k$. Then 
 \[ \Hilb_{S/J_{G}}(t)= \underset{ i \in [k]} \sum \frac{1}{(1-t)^{2p_i}}-
 \underset{ i \in [k]} \sum \frac{(p_i-1)t+1}{(1-t)^{p_i +1}}+ \frac{(n-1)t+1}{(1-t)^{n+1}} .\] In particular, $\dim(S/J_G)=\max\{2p_1, n+1  \}$ and
 \[
e(S/J_{G}) = 
     \begin{cases}
      {n} & \quad\text{if p}_1=1 \quad\text{or } 2p_1 <n+1, \\
       {1} &\quad\text{if }2p_1 > n+1, \\
         {n+1} &\quad\text{if }2p_1 = n+1. 
     \end{cases}
\]
\end{corollary}
\begin{proof}
Note that $G= (\cdots(K_{p_1}^c*K_{p_2}^c)*\cdots)* K_{p_r}^c$. Now the result follows by recursively applying Theorem \ref{4.15}.
\end{proof}
 
 %\nocite*{}
\bibliographystyle{plain}  %% or 
\bibliography{binomial}

\end{document}